\documentclass[12pt]{amsart}
\usepackage{amsmath,amssymb,amsbsy,amsfonts,latexsym,amsopn,amstext,cite,
                                               amsxtra,euscript,amscd,bm}
\usepackage{url}

\usepackage{mathrsfs}
\usepackage{color}
\usepackage[colorlinks,linkcolor=blue,anchorcolor=blue,citecolor=blue,backref=page]{hyperref}
\usepackage{color}
\usepackage{graphics,epsfig}
\usepackage{graphicx}
\usepackage{float}
\usepackage{epstopdf}
\hypersetup{breaklinks=true}

\usepackage[np]{numprint}
\npdecimalsign{\ensuremath{.}}

\usepackage{bibentry}

\usepackage[english]{babel}
\usepackage{mathtools}
\usepackage{todonotes}
\usepackage{url}
\usepackage[colorlinks,linkcolor=blue,anchorcolor=blue,citecolor=blue,backref=page]{hyperref}

\usepackage[norefs,nocites]{refcheck}

\DeclareMathOperator*\lowlim{\underline{lim}}

\DeclareMathOperator*\uplim{\overline{lim}}

\begin{document}

\newtheorem{theorem}{Theorem}
\newtheorem{lemma}[theorem]{Lemma}
\newtheorem{claim}[theorem]{Claim}
\newtheorem{cor}[theorem]{Corollary}
\newtheorem{conj}[theorem]{Conjecture}
\newtheorem{prop}[theorem]{Proposition}
\newtheorem{definition}[theorem]{Definition}
\newtheorem{constr}[theorem]{Construction}
\newtheorem{question}[theorem]{Question}
\newtheorem{example}[theorem]{Example}
\newcommand{\hh}{{{\mathrm h}}}
\newtheorem{remark}[theorem]{Remark}

\numberwithin{equation}{section}
\numberwithin{theorem}{section}
\numberwithin{table}{section}
\numberwithin{figure}{section}

\def\sssum{\mathop{\sum\!\sum\!\sum}}
\def\ssum{\mathop{\sum\ldots \sum}}
\def\iint{\mathop{\int\ldots \int}}

\newcommand{\diam}{\operatorname{diam}}

\def\squareforqed{\hbox{\rlap{$\sqcap$}$\sqcup$}}
\def\qed{\ifmmode\squareforqed\else{\unskip\nobreak\hfil
\penalty50\hskip1em \nobreak\hfil\squareforqed
\parfillskip=0pt\finalhyphendemerits=0\endgraf}\fi}

\newfont{\teneufm}{eufm10}
\newfont{\seveneufm}{eufm7}
\newfont{\fiveeufm}{eufm5}
%
%
\newfam\eufmfam
     \textfont\eufmfam=\teneufm
\scriptfont\eufmfam=\seveneufm
     \scriptscriptfont\eufmfam=\fiveeufm
%
%
\def\frak#1{{\fam\eufmfam\relax#1}}

\newcommand{\bflambda}{{\boldsymbol{\lambda}}}
\newcommand{\bfmu}{{\boldsymbol{\mu}}}
\newcommand{\bfxi}{{\boldsymbol{\eta}}}
\newcommand{\bfrho}{{\boldsymbol{\rho}}}
\newcommand{\bfeps}{{\boldsymbol{\eps}}}

\def\eps{\varepsilon}

\def\fK{\mathfrak K}
\def\fT{\mathfrak{T}}
\def\fL{\mathfrak L}
\def\fR{\mathfrak R}

\def\fA{{\mathfrak A}}
\def\fB{{\mathfrak B}}
\def\fC{{\mathfrak C}}
\def\fM{{\mathfrak M}}
\def\fS{{\mathfrak  S}}
\def\fU{{\mathfrak U}}
\def\fZ{{\mathfrak Z}}
\def\fW{{\mathfrak W}}

\def\T {\mathsf {T}}
\def\Tor{\mathsf{T}_d}
\def\Tore{\widetilde{\mathrm{T}}_{d} }

\def\sM {\mathsf {M}}

\def\Kmnd{\cK_d(m,n)}
\def\Kmnp{\cK_p(m,n)}
\def\Kmnq{\cK_q(m,n)}

\def \balpha{\bm{\alpha}}
\def \bbeta{\bm{\beta}}
\def \bgamma{\bm{\gamma}}
\def \bdelta{\bm{\delta}}
\def \bzeta{\bm{\zeta}}
\def \blambda{\bm{\lambda}}
\def \bchi{\bm{\chi}}
\def \bphi{\bm{\varphi}}
\def \bpsi{\bm{\psi}}
\def \bnu{\bm{\nu}}
\def \bomega{\bm{\omega}}

\def \bell{\bm{\ell}}

\def\eqref#1{(\ref{#1})}

\def\vec#1{\mathbf{#1}}

\newcommand{\abs}[1]{\left| #1 \right|}

\def\CC{\mathbb{C}}
\def\Zq{\mathbb{Z}_q}
\def\Zqx{\mathbb{Z}_q^*}
\def\Zd{\mathbb{Z}_d}
\def\Zdx{\mathbb{Z}_d^*}
\def\Zf{\mathbb{Z}_f}
\def\Zfx{\mathbb{Z}_f^*}
\def\Zp{\mathbb{Z}_p}
\def\Zpx{\mathbb{Z}_p^*}
\def\cM{\mathcal M}
\def\cE{\mathcal E}
\def\cH{\mathcal H}

\def\le{\leqslant}

\def\ge{\geqslant}

\def\sfB{\mathsf {B}}
\def\sfC{\mathsf {C}}
\def\L{\mathsf {L}}
\def\FF{\mathsf {F}}

\def\sE {\mathscr{E}}
\def\sS {\mathscr{S}}

\def\cA{{\mathcal A}}
\def\cB{{\mathcal B}}
\def\cC{{\mathcal C}}
\def\cD{{\mathcal D}}
\def\cE{{\mathcal E}}
\def\cF{{\mathcal F}}
\def\cG{{\mathcal G}}
\def\cH{{\mathcal H}}
\def\cI{{\mathcal I}}
\def\cJ{{\mathcal J}}
\def\cK{{\mathcal K}}
\def\cL{{\mathcal L}}
\def\cM{{\mathcal M}}
\def\cN{{\mathcal N}}
\def\cO{{\mathcal O}}
\def\cP{{\mathcal P}}
\def\cQ{{\mathcal Q}}
\def\cR{{\mathcal R}}
\def\cS{{\mathcal S}}
\def\cT{{\mathcal T}}
\def\cU{{\mathcal U}}
\def\cV{{\mathcal V}}
\def\cW{{\mathcal W}}
\def\cX{{\mathcal X}}
\def\cY{{\mathcal Y}}
\def\cZ{{\mathcal Z}}
\newcommand{\rmod}[1]{\: \mbox{mod} \: #1}

\def\cg{{\mathcal g}}

\def\vy{\mathbf y}
\def\vr{\mathbf r}
\def\vx{\mathbf x}
\def\va{\mathbf a}
\def\vb{\mathbf b}
\def\vc{\mathbf c}
\def\vh{\mathbf h}
\def\vk{\mathbf k}
\def\vm{\mathbf m}
\def\vz{\mathbf z}
\def\vu{\mathbf u}
\def\vv{\mathbf v}

\def\e{{\mathbf{\,e}}}
\def\ep{{\mathbf{\,e}}_p}
\def\eq{{\mathbf{\,e}}_q}

\def\Tr{{\mathrm{Tr}}}
\def\Nm{{\mathrm{Nm}}}

 \def\SS{{\mathbf{S}}}

\def\lcm{{\mathrm{lcm}}}

 \def\0{{\mathbf{0}}}

\def\({\left(}
\def\){\right)}
\def\l|{\left|}
\def\r|{\right|}
\def\fl#1{\left\lfloor#1\right\rfloor}
\def\rf#1{\left\lceil#1\right\rceil}
\def\sumstar#1{\mathop{\sum\vphantom|^{\!\!*}\,}_{#1}}

\def\mand{\qquad \mbox{and} \qquad}

\def\tblue#1{\begin{color}{blue}{{#1}}\end{color}}




\hyphenation{re-pub-lished}

\mathsurround=1pt

\def\bfdefault{b}

\def \F{{\mathbb F}}
\def \K{{\mathbb K}}
\def \N{{\mathbb N}}
\def \Z{{\mathbb Z}}
\def \Q{{\mathbb Q}}
\def \R{{\mathbb R}}
\def \C{{\mathbb C}}
\def\Fp{\F_p}
\def \fp{\Fp^*}

 \def \xbar{\overline x}


\title[Small values of  Weyl sums]{Small values  of  Weyl sums}

 \author[C. Chen] {Changhao Chen}

\address{Department of Pure Mathematics, University of New South Wales,
Sydney, NSW 2052, Australia}
\email{changhao.chenm@gmail.com}

 \author[I. E. Shparlinski] {Igor E. Shparlinski}

\address{Department of Pure Mathematics, University of New South Wales,
Sydney, NSW 2052, Australia}
\email{igor.shparlinski@unsw.edu.au}

\begin{abstract}  
We prove that the set of $(x_1, \ldots, x_d)\in [0,1)^d$,  such that    
$$
\lowlim_{N\to \infty}\left| \sum_{n=1}^N\exp(2 \pi i (x_1n+\ldots + x_dn^d)) \right| =0,
$$
contains a  dense $\cG_\delta$  set in $[0,1)^d$ and has positive Hausdorff dimension.
Similar statements are also established for the generalised Gaussian sums 
$$
\sum_{n=1}^N\exp(2\pi i x n^d), \qquad x \in [0,1). 
$$
\end{abstract}



\maketitle


\section{Introduction}
\subsection{Motivation}
For an integer $d \geqslant 2$, let $\Tor = (\R/\Z)^d$ denote  the  $d$-dimensional unit torus.  For  a vector $\vx = (x_1, \ldots, x_d)\in \Tor$ and $N \in\N$, we consider the {\it Weyl sums\/}
$$
S_d(\vx; N)=\sum_{n=1}^{N}\e\(x_1 n+\ldots +x_d n^{d} \), 
$$
where   throughout  the paper we denote $\e(x)=\exp(2 \pi i x)$.  
Since we also consider   {\it monomial sums\/}
$$
G_d(x;N)=\sum_{n=1}^{N} e(xn^{d}), 
$$
it  is also convenient to define 
$$
\T = \T_1 = \R/\Z.
$$
We recall that for $d=2$ the sums $ G_2(x;N)$ are also known as  {\it Gaussian sums\/}
and we denote
$$
G(x;N)= G_2(x;N).
$$

Obtaining upper bounds on these sums has received a lot attention 
over the last several decades, most significantly due the the proof of the called {\it main conjecture
in the Vinogradov  mean value theorem\/} by  Bourgain, Demeter and Guth~\cite{BDG} (for $d \geqslant 4$) 
and Wooley~\cite{Wool2} (for $d=3$)  (see also~\cite{Wool5}). In particular,   the state of art 
is conveniently summarised by Bourgain~\cite[Theorem~5]{Bourg}. 

On the other hand, the problem of the distribution of values and in particular of lower bounds seems
to be less know with just several sporadic results in the literature, mostly related to the case $d=2$, and 
those do not seem to be widely known. Motivated by this, we first give a survey of known results 
about the measure and topological properties of sets of $\vx  \in \Tor$ and $x \in \T$  large 
sums $S_d(\vx; N)$ and $G_d(x; N)$
and then obtain two new  results.

\subsection{A survey of results on the distribution of Weyl sums}
We first show some known results for the case $d=2$.   The study of the sums 
$S_2(\vx; N)$ and $G(x; N)$
has been initiated by Hardy and Littlewood~\cite{HL}. 

 First we recall  we say  that  $x\in T$ has bounded partial quotients if $\sup_{n\in \N} a_n <\infty$, where 
$x =  [a_1, a_2, \ldots]$ is  the continued fraction representation of $x$.
Hardy and Littlewood~\cite{HL} have given the following lower and upper bounds.

\begin{theorem} {\rm(Hardy and Littlewood~\cite[Theorems~2.22 and~2.25]{HL})}
\label{thm:HL}
Let $x\in \T$ be irrational. Then  
\begin{itemize}
\item[(i)] there exists a constant $c>0$ such that 
$$
\left |G(x; N)\right | \ge c \sqrt{N}
$$
for infinitely many $N\in \N$; 

\item[(ii)] if the continued fraction of $x$ has bounded partial quotients,  then there exist absolute constants  $C>c > 0$ such that for any $y\in \T$ one has 
$$
c\sqrt{N}\le  \left |S_2\((y,x); N\)\right | \le C \sqrt{N}.
$$
\end{itemize}
\end{theorem}

Furthermore, except for  a set of $x\in \T$ of  Lebesgue measure  zero, Fiedler, Jurkat and K\"orner~\cite[Theorem~2]{FJK} 
have obtained the optimal bound for the sums $G(x; N)$.

We say that some property holds for almost all $\vx \in \Tor$ if it holds for a set 
 $\cX \subseteq \Tor$ of  Lebesgue measure  $\lambda(\cX) = 1$.   

\begin{theorem} {\rm(Fiedler, Jurkat and K\"orner~\cite[Theorem~2]{FJK})} 
\label{thm:FJK}
Suppose that  $\{f(n)\}_{n=1}^{\infty}$ is a non-decreasing sequence of positive numbers. Then for almost all  $x\in \T$  one has 
$$ 
\uplim_{N\to  \infty} \frac{ \left |G(x; N)\right |}{\sqrt{N} f(N)}<\infty\quad  \Longleftrightarrow \quad \sum_{n=1}^{\infty} \frac{1}{n f(n)^{4}} <\infty.
$$
\end{theorem}  

For $x\in [0,1]$ instead of obtaining the bound of the sums 
$\sum_{n=1}^N \e(n^2x)$, Jurkat and van Horne \cite{JH1, JH2, JH3} studied  the sequence of the distribution function 
$$
\Psi_N(\alpha)=\lambda (\{x\in [0,1]: N^{-1/2} |\sum_{n=1}^N \e(n^2x) |\ge \alpha\}),
$$
where $\lambda$ is the Lebesgue measure. Among other things Jurkat and van Horne \cite{JH1, JH2, JH3} proved that $\Psi_N$ convergences to a limiting distribution (the limit is not normal distribution).

Marklof~\cite{Marklof} has studied the  asymptotic behaviour of the sums 
$$
N^{-1/2} \sum_{n=1}^N \e(n^2x) 
$$
in the complex plane $\C$ as $N\rightarrow \infty$ for all $x\in \R$.

We remark that the methods of Jurkat and van Horne \cite{JH1, JH2, JH3} are mainly based on the circle method, Diophantine approximations, and  bounds of  Kloosterman sums, while the approach  of Marklof~\cite{Marklof} stems from
the theory of dynamical systems.

For quadratic
 Weyl sums, a generalisation of Theorem~\ref{thm:FJK} is given by Fedotov and Klopp~\cite[Theorem~0.2]{FK}). Furthermore, Fedotov and Klopp~\cite{FK} have given a similar result   for the sums $S_2(\vx; N)$, however 
adding the term $\e(x_1n)$ leads to more cancellations in the sums $S_2(\vx; N)$.

\begin{theorem} {\rm(Fedotov and Klopp~\cite[Theorem~0.1]{FK})}
\label{thm:FK}
Suppose that  $\{g(n)\}_{n=1}^{\infty}$ is a non-decreasing sequence of positive numbers. Then for almost all  $\vx\in \T_2$  one has 
$$
\uplim_{N\to  \infty} \frac{  \left |S_2(\vx; N)\right |}{\sqrt{N} g(\ln  N)}<\infty \quad  \Longleftrightarrow \quad  \sum_{n=1}^{\infty} \frac{1}{ g(n)^{6}} <\infty.
$$
\end{theorem}

For the case $d\ge 3$, the authors in~\cite[Appendix~A]{ChSh1},  and~\cite[Theorem~2.1]{ChSh2} have  shown in two different ways  that  for almost all $\vx\in \Tor$  one has   
 \begin{equation}
\label{eq:M-R}
  \left |S_d(\vx; N)\right |\leqslant  N^{1/2+o(1)} \quad \text{as} \quad N\to  \infty.
\end{equation}

It is natural to conjecture that the bound~\eqref{eq:M-R}  is close to 
the best possible up to the $N^{o(1)}$ term, see also~\cite[Conjecture~1.1]{ChSh1}.  We remark that the conjecture is still open for the case $d\ge 3$.

Observe that we may consider the sums $S_d(\vx; N)$,  $N\in \N$, as a sequence of points in the complex plane $\CC$. Before giving  some results in this direction we show some notation first.

We now need to recall some standard definitions.

\begin{definition} We sat that a set $\cS$ of a topological space $\cX$ is a  $\cG_{\delta}$-set  
if it is a countable intersection of open sets. 
\end{definition}

We could say that a dense $\cG_{\delta}$  set is a ``large'' set in the sense of topology. We remark that the $\cG_{\delta}$  set is closely related to Baire categories, see~\cite[Section~9]{Oxtoby} for more details.



%
%
%

Given $x\in \T$ define 
$$
\cB(x)=\{y\in \T:~S_2\((y,x); N\), \ N\in \N, \text{ is dense in } \C\}.
$$

For a real $x$  the notation we use $\| x\|$  to denote the distance to the closest  integers.

Using  an approach which stems from the theory of dynamical systems and considering the Weyl sums as a cocycle on $\CC$, Forrest~\cite{Forrest1}
have obtained the following result.


\begin{theorem}{\rm(Forrest~\cite[Theorem~2]{Forrest1})}
\label{thm:Forrest}
If $x\in \T$ is  irrational  and 
$$
\lowlim _{q\to  \infty}q^{3/2}\| q x\|<\infty,
$$ 
then  $\cB(x)$ contains a  dense $\cG_{\delta}$ set in $\T$.  
\end{theorem}

There are serval  interesting generations of Theorem~\ref{thm:Forrest}. 

\begin{theorem}{\rm(Forrest~\cite[Theorem~1.3]{Forrest2})}
Suppose that $x\in \T$ is  irrational  and has continued fraction representation $[a_1, a_2, \ldots]$ such that 
$\sum_{n=1}^\infty 1/a_n <\infty$, and suppose that 
$$
\lowlim _{q\to  \infty}q^{3+\eps}\| q x\|=0,
$$ 
for some $\eps>0$. Then $\cB(x)$ is of full Lebesgue measure in $\T$.
\end{theorem}

We remark that Fayad~\cite{Fayad} and Greschonig, Nerurkar and Voln\'y~\cite{GNV} studied the ergodic property of the dynamical system which was introduced by Forrest~\cite{Forrest1, Forrest2} (for the purpose of studying the distribution of Weyl sums). 
We refer to~\cite{Fayad, GNV} for more details and reference therein.

We remark that the following conjecture of Forrest~\cite{Forrest1} is still open. 

\begin{conj}{\rm(Forrest~\cite{Forrest1})}
\label{conj:Forr} 
Theorem~\ref{thm:Forrest} is true for every irrational $x$ with 
$\lowlim _{q\to  \infty}q\| q x\|=0$.
\end{conj}

 Forrest~\cite{Forrest1} also considered the sequence $S_d(\vx; N), N\in \N$ for the case $d\ge 3$. In analogy of  Theorem~\ref{thm:Forrest}, Forrest~\cite{Forrest1} showed the following.  

\begin{theorem}{\rm(Forrest~\cite[Theorem~10]{Forrest1})}
Suppose that $d\ge 3$ and $x_d\in \T$ is irrational such that 
$$
\lowlim _{q\to  \infty}q^{d}\| q x_d\|<\infty,
$$ 
then for a dense $G_\delta$ set of $(x_1, \ldots, x_{d-1})\in [0,1]^{d-1}$, the sequence of partial sums  
$$
S_d(\vx; N) , \quad N=1, 2, \ldots,
$$
is dense in $\mathbb{C}$.
\end{theorem}

With the weaker condition for the leading term $x_d$, Forrest~\cite{Forrest1} gives  the following
result.

\begin{theorem}{\rm(Forrest~\cite[Proposition~13]{Forrest1})}
\label{thm:Forrest-dense-zero}
Suppose that $d\ge 3$ and $x_d$ is irrational such that 
$$
\lowlim _{q\to  \infty}q^{d-1/2}\| q x_d\|<\infty,
$$ 
then for a dense $G_\delta$ set of  $(x_1, \ldots, x_{d-1})\in \T_{d-1}$ the partial sums 
$$
S_d(\vx; N), \quad N=1, 2, \ldots,
$$
is dense at $0$ in  $\mathbb{C}$.
\end{theorem}

One may also conjecture that for almost all $\vx\in \T_d$  the sequence $S_d(\vx; N), N\in \N$ is dense in $\CC$, see Conjecture~\ref{con:dense} below.

\subsection{Main results} 
We now show that there are many very small  quadratic Weyl sums and monomial sums for arbitrary degree.  More precisely we define the ``zero sets'' of   Weyl sums and monomial sums  as 
$$
\cZ_{d,W}=\{\vx\in \Tor:~\lowlim_{N\to \infty} |S_d(\vx; N)|=0\},
$$
and
$$
\cZ_{d,G}=\{x\in \T:~\lowlim_{N\to \infty} |G_d(x; N)|=0\}.
$$
We now show that  $\cZ_{d,W}$ and $\cZ_{d,G}$  are quite massive. 

\begin{theorem} 
\label{thm:zero}
For any integer $d\ge 2$,  the sets  $\cZ_{d,W}$  and $ \cZ_{d,G}$  contain dense $\cG_{\delta}$ sets in $\Tor$ and $\T$, respectively.
\end{theorem} 

We remark that the claim of Theorem~\ref{thm:zero} for  $\cZ_{d,W}$ can perhaps be derived from 
Theorem~\ref{thm:Forrest-dense-zero}, however we  prove it via  a completely different method, which also 
applies to  $ \cZ_{d,G}$ and perhaps to some other similar sets. 

Note  that if $\cA, \cB$ contain dense $\cG_{\delta}$ sets, then their intersection $\cA\cap \cB$  also contains a dense $\cG_{\delta}$ set. 

We remark that the proof of~\cite[Theorem~1.3]{ChSh1}  implies that for  any integer $d\ge 2$ the set 
$\Xi_{d}$ contains a dense $\cG_{\delta}$ set, where  
\begin{align*}
\Xi_{d}= \bigl\{\vx\in \Tor:~\forall \varepsilon>0, \,\uplim_{N\to  \infty} \frac{|S_d(\vx; N)|}{N^{1-\varepsilon}} =\infty \bigr\}.
\end{align*}
Meanwhile, the proof of~\cite[Theorem~1.6]{ChSh1} implies that the similar statement also holds for the monomial sums $G_d(x; N)$ with any integer $d\ge 2$.   Therefore  for integer $d\ge 2$, we conclude that there are dense $\cG_{\delta}$ sets of  $\vx\in \T_d$ and $x\in T$  with  
\begin{equation}
\label{eq:zero-accumulate}
\lowlim_{N\to \infty} |S_d(\vx; N)|=0 \mand \lowlim_{N\to \infty} |G_d(x; N)|=0,
\end{equation}
and for any $\varepsilon>0$ 
\begin{equation}
\label{eq:infinity-accumulate}
\uplim_{N\to \infty} \frac{|S_d(\vx; N)|}{N^{1-\eps}}=\infty \mand  \uplim_{N\to \infty} \frac{|G_d(x; N)|}{N^{1-\eps}}=\infty,
\end{equation}
 respectively. 

 It is natural to ask about the Lebesgue measure and the  Hausdorff dimensions of the sets  $\cZ_{d,W}$ and $\cZ_{d,G}$.

\begin{definition}[Hausdorff dimension]
The  Hausdorff dimension of a set $\cA\subseteq \R^{d}$ is defined as 
\begin{align*}
\dim_H \cA=\inf\Bigl\{s>0:~\forall \,   \eps>0,&~\exists \, \{ \cU_i \}_{i=1}^{\infty}, \ \cU_i \subseteq \R^{d},\  \text{such that }  \\
  &  \cA\subseteq \bigcup_{i=1}^{\infty} \cU_i \text{ and } \sum_{i=1}^{\infty}\(\diam\cU_i\)^{s}<\eps \Bigr\}.
\end{align*}
\end{definition} 

For the properties of the Hausdorff dimension and its applications we refer to~\cite{Falconer, Mattila1995}. 

We now give lower bounds on the Hausdorff dimensions of $  \cZ_{d,W}$ and   $\cZ_{d,G}$.

\begin{theorem} 
\label{thm:H-dim}
For any integer $d\ge 2$,  we have
$$ 
\dim_H \cZ_{d,W} \ge 6/5 \mand 
\dim_H \cZ_{d,G} \ge 4/(2d+1).
$$
\end{theorem}

The above lower bounds on the Hausdorff dimension are perhaps  far from the truth but we do not know how to improve them.

\subsection{Outline of the method} We first show that the quadratic Weyl sums and monomial sums  take  small values at some  rational points and thus in their neighbourhood. Secondly, we apply some results and tools from metric number theory to show that the neighbourhoods of these rational points are large sets in the sense of  topology and Hausdorff dimension.

We conclude this paper with a brief outline of some  further directions of research and some additional 
ideas which may lead to improvements of our results. In particular, we formulate several conjectures on the
behaviour of Weyl sums.

\section{Exponential sums}

\subsection{Notation and conventions}

As usual, the notations $U = O(V )$, 
$U \ll V$ and $ V\gg U$  are equivalent to $|U|\leqslant c|V| $ for some positive constant $c$.
Throughout the paper, all implied constants are absolute. 
 

 The letter $p$ always denotes a prime number. 
 
\subsection{Complete and incomplete Gaussian  sums} 
For a prime $p$,  let $\F_p$ denote the finite field of $p$ elements, which we identify with the set
$\{0, \ldots, p-1\}$. Furthermore, let 
$$\e_m(z)=\e(z/m).
$$

We also need the following bound for the incomplete Gaussian sums, see~\cite{K}.

\begin{lemma} 
\label{lem:K}
For each prime $p$ and integer $b\neq 0$ we have 
$$
\max_{1\le N\le p}\left |\sum_{n=1}^N\e_p(bn^2)\right | \ll \sqrt{p}.
$$
\end{lemma}

From Lemma~\ref{lem:K} we immediately  derive the  following.

\begin{lemma}
\label{lem:incompleteGauss}
For any prime  $p$ and any $a, b\in \F_p$ with $b\neq 0$ we have 
$$
\max_{1\le M, N\le p}\left|\sum_{M+1\le n\le M+N}\ep\left(an+bn^{2}\right)\right| \ll \sqrt{p}.
$$
\end{lemma} 

We emphasise that the implied constant in  Lemma~\ref{lem:incompleteGauss} is absolute. 

We now recall that the Gaussian sum modulo $2p$ or $4p$ exhibit  very different behaviour 
and sometimes vanish. 

\begin{lemma}
\label{lem:zero}
Let $p\ge 3$ and $a, b\in \Z$. 
\begin{itemize}
\item[(i)] If $\gcd(b, 2p)=1$ then 
$$
\sum_{n=1}^{2p} \e_{2p} \(bn^2\)=0.
$$
\item[(ii)]  If $\gcd(ab, 2p)=1$ then  
$$
\sum_{n=1}^{4p} \e_{4p}\(an+bn^2\)=0.
$$
\end{itemize}
\end{lemma}

\begin{proof} Part~(i) is well known and follows instantly from the explicit   formula for Gauss sums,  see~\cite[Theorem~3.4]{IwKow}.

For Part~(ii), we note 
\begin{align*}
\sum_{n=1}^{4p} \e_{4p}\(an+bn^2\) &= \sum_{n=1}^{4p} \e_{4p}\(a(n+2p)+b(n+2p)^2\)\\
& = \sum_{n=1}^{4p} \e_{4p}\(an+bn^2 + 2ap+4pn + 4p^2\)\\
& = \sum_{n=1}^{4p} \e_{4p}\(an+bn^2 + 2ap\)\\
& = - \sum_{n=1}^{4p} \e_{4p}\(an+bn^2\)
\end{align*}
since for an odd $a$ we have
$  \e_{4p}\(2ap\) = \exp(\pi i a) = -1$.
 The result now follows. 
\end{proof}

\subsection{Monomial  sums} 
Here we always assume that $d \ge 3$.  We  have the following analogues of Lemma~\ref{lem:zero}, which follows from the trivial observation that 
the map $x \mapsto x^d$ is a permutation of $\F_p$ provided that  $\gcd(d, p-1)=1$.

\begin{lemma}
\label{lem:zero-m} 
Let  an integer $d\ge 3$ and  a prime $p\ge 3$ be such that $\gcd(d, p-1)=1$. Then for any integer $b$ with $\gcd(b, p)=1$ one has 
$$
\sum_{n=1}^{p} \e\left(\frac{bn^d}{p}\right)=0.
$$
\end{lemma}

Using the Weil bound together with the standard 
completion technique, see~\cite[Sections~11.11 and~12.2]{IwKow} we also immediately obtain: 

\begin{lemma}
\label{lem:incompleteMonom}
For any  prime  $p$ and any $a \in \F_p\setminus\{0\}$ we have 
$$
\max_{1\le M, N\le p}\left|\sum_{M+1\le n\le M+N}\ep\left(an^{d}\right)\right| \ll \sqrt{p} \log p.
$$
\end{lemma} 

We again emphasise that the implied constant in  Lemma~\ref{lem:incompleteMonom} is absolute.

\subsection{Continuity of Weyl sums}

We present our next result in a much more general form than we need for 
purpose, but we believe in this generality it may have other applications.

\begin{lemma}   
\label{lem:con gen} Let integer $N \ge 1$ and a vector $\vx \in \Tor$ be such that 
for any $M \le N$ we have 
$$
S_d(\vx; M) \ll \kappa M^\alpha + K
$$
for some real non-negative $\alpha$,  $\kappa$ and $K$.  Then for any  positive  $\tau = O(1)$ and  $\vy \in \Tor$ 
with 
$$
0\le y_i-x_i <\tau N^{-i}, \qquad i =1, \ldots, d,
$$
we have 
$$
S_d(\vy; N) - S_d(\vx; N)\ll \tau\( \kappa N^\alpha + K\),
$$
where the implied constant is absolute. 
\end{lemma}

\begin{proof}
Let  $\delta_i = y_i - x_i$,  $i =1, \ldots, d$. For each $n\in \N$ we have 
\begin{align*}
\e\(y_1 n+\ldots +y_d n^{d} \)&=\e\(x_1 n+\ldots +x_d n^{d} \)\e(\delta_1n+\ldots + \delta_dn^d)\\
&=\e\(x_1 n+\ldots +x_d n^{d} \)\\
& \qquad \qquad \qquad \sum_{k=0}^{\infty} \frac{(2\pi i (\delta_1n+\ldots + \delta_dn^d))^{k}}{k!}.
\end{align*}

It follows that 
\begin{equation} \label{eq:diff}
\begin{split}
S_d(\vy; N)- S_d(\vx; N) &  = \sum_{k=1}^{\infty} \sum_{n=1}^{N} \e\(x_1 n+\ldots +x_d n^{d} \) \\
& \qquad \qquad \qquad 
\frac{(2\pi i  (\delta_1n+\ldots + \delta_dn^d))^{k}}{k!}.
\end{split}
\end{equation}
For each $k\in \N$ we now turn to the estimate 
$$
\sigma_k =\sum_{n=1}^{N} \e\(x_1 n+\ldots +x_d n^{d} \)  \xi_{n}^{k}.
$$
where $\xi_{n}=\delta_1n+\ldots + \delta_dn^d$.   Applying partial sum formula we derive 
$$
\sigma_k=  \sigma_{k,1} + \sigma_{k,2}
$$
where 
$$
\sigma_{k,1} = S_d(\vx; N) \xi_{N}^k  \mand  \sigma_{k,2}
 = \sum_{M=1}^{N-1} S_d(\vu; M)\(\xi_{M}^k -\xi_{M+1}^k\).
$$
By our assumption,  we obtain 
\begin{equation} \label{eq:sigma1}
\sigma_{k,1}\ll ( \kappa N^\alpha + K)(\delta_1N+\ldots + \delta_dN^d)^k \ll (d \tau)^k  ( \kappa N^\alpha + K)
\end{equation}
and also, observing that the sequence $\xi_{M}$ is monotonically non-decre\-asing, we have
\begin{align*}
 \sigma_{k,2} & \ll  \sum_{M=1}^{N-1} ( \kappa M^\alpha + K)   \left| \xi_{M}^k -\xi_{M+1}^k\right|\\
&  =  \kappa  \sum_{M=1}^{N-1} M^\alpha \(\xi_{M+1}^k -\xi_{M}^k\) 
+ K\sum_{M=1}^{N-1} \(\xi_{M+1}^k -\xi_{M}^k\).  
\end{align*}

We now derive
\begin{equation} \label{eq:sigma2}
\begin{split}
 \sigma_{k,2} &\ll   \kappa  \sum_{M=1}^{N} M^{\alpha-1} \xi_{M}^k
+ K \xi_{N}^k \\
& \ll  \kappa   \sum_{M=1}^{N} M^{\alpha-1}  (d \tau)^k 
+ K (d \tau)^k\\
&  \ll (d \tau)^k  ( \kappa N^\alpha + K). 
\end{split}
\end{equation}

We see from~\eqref{eq:sigma1} and~\eqref{eq:sigma2} that 
$$
 \sigma_k \ll    (d \tau)^k  ( \kappa N^\alpha + K)
$$
which together with~\eqref{eq:diff}  yields the desired bound. 
\end{proof}

For the convenience of our applications we   formulate several specialisations of Lemma~\ref{lem:con gen}  
with different choices of $(\alpha, \kappa, K)$.  

For example,  with $(\alpha, \kappa, K)=(0, 0, K)$ we obtain the following result on approximations to rational sums. 

\begin{cor}
\label{cor:for-using} Let integer $N \ge 1$ and integer $m\ge 2$. Let integer $d\ge 3$ and 
 $$
 \va =(a_1, \dots, a_d)\in (\Z_m \setminus\{0\})^d
 $$ 
 be such that for any $n\le N$ we have 
$$
S_d(\va/m; n) \ll K
$$
for some real non-negative $K$.  Then for any  $\tau>0$ and  $\vx \in \Tor$ 
with 
$$
0\le x_j-a_j/m <\tau N^{-j}, \qquad j=1, \ldots, d,
$$
we have 
$$
S_d(\vx; N) - S_d(\va/m; N)\ll \tau K,
$$
where the implied constant is absolute. 
\end{cor}

 If $d=2$ and $\vx = (a/p, b/p) \in \T_2$ for some  integers $a$ and $b$ with $\gcd(b, p) =1$, then by Lemma~\ref{lem:incompleteGauss} we can take  $(\alpha, \kappa, K)= (1, p^{-1/2}, p^{1/2})$
 in Lemma~\ref{lem:con gen}.    

\section{Proofs of  results on small sums}

\subsection{Proof of Theorem~\ref{thm:zero}}

We first show the sets $ \cZ_{2, W},  \cZ_{2, G}$ and $\cZ_{d, G}$ (with $d\ge 3$) are dense $\cG_{\delta}$ sets. We introduce some notation. 
For $p\ge 3$ let 
$$
\cP_{p}=\left\{\(\frac{a}{4p}, \frac{b}{4p}\):~\gcd(ab, 2p)=1,\ 1\le a, b\le 4p,\ a, b\in \N \right\}.
$$

Furthermore, let 
$$
\cQ_{p}=\left \{\frac{b}{2p}:~\gcd(b, 2p)=1, \ 1\le b\le 2p,\ b\in \N \right\} 
$$
and 
$$
\cR_{p}=\left \{\frac{b}{p}:~ 1\le b\le p\right\}.
$$

For $\cA\subseteq \R^d$ the $\delta$-neighbourhood of  $\cA$ is defined as 
\begin{equation}
\label{eq:A}
\cA(\delta)=\{\vx\in \R^d:~\exists\, \va \in \cA \, \text{ such that}\,  |\va-\vx|<\delta\}.
\end{equation}

Clearly for any $p\ge 3$  the functions $S_2(\vx; 4p)$,  $G(x; 2p)$ and $G_d(x; p)$
are continuous function with respect to the variables $\vx\in \T_2$ and  $x\in \T$.  Thus, applying Lemmas~\ref{lem:zero} and~\ref{lem:zero-m}, we derive that for any  $p$ and $\eta>0$ there exists $\delta_{p, \eta} >0$ such that
\begin{itemize}
\item for any  $\vx\in \cP_p(\delta_{p, \eta} )$ we have $|S_2(\vx; 4p)|\le \eta$;
 \item  for $d=2$ and  any  $x\in \ \cQ_p(\delta_{p, \eta} )$   we have   $|G(x; 2p)|\le \eta$;
  \item  for $d\ge 3$, $\gcd(d,p-1)=1$ and   any  $x\in  \cR_p(\delta_{p, \eta} )$   we have   $|G_d(x; p)|\le \eta$. 
\end{itemize}

Using these notation and~\eqref{eq:A}, we set
$$
\widetilde{P} =\bigcap_{j=1}^{\infty}\bigcap_{k=1}^{\infty}\bigcup_{p \geqslant k} \cP_{ p}(\delta_{p, 1/j})\mand 
 \widetilde{Q}=\bigcap_{j=1}^{\infty}\bigcap_{k=1}^{\infty}\bigcup_{p \geqslant k}  \cQ_{ p}(\delta_{p,1/j}), 
$$
and also for $d \ge 3$ 
$$
\widetilde{R}=\bigcap_{j=1}^{\infty}\bigcap_{k=1}^{\infty}\bigcup_{\substack {p \geqslant k\\ \gcd( d,p-1)=1}} \cR_{ p}(\delta_{p,1/j}).
$$

Observe that 
$$
\widetilde{P}\subseteq \cZ_{2,W},  \quad\widetilde{Q}\subseteq \cZ_{2,G}, \quad \widetilde{R}\subseteq \cZ_{d, G} \ (d \ge 3), $$
and for any $k$ the set 
$$
 \bigcup_{p \geqslant k}   \cP_p \subseteq \T_2, 
 \qquad   \bigcup_{p \geqslant k}   \cQ_p \subseteq \T,  \qquad \bigcup_{\substack{p \geqslant k \\ \gcd( d,p-1)=1} }   \cR_p \subseteq \T, 
 $$
are dense open sets, which  finishes the proof for the sets $ \cZ_{2, W},  \cZ_{2, G}$ and $ \cZ_{d, G}$ (with $d\ge 3$).

We now turn to the claim that the set  $\cZ_{d, W}$ with $d\ge 3$ is a dense $\cG_{\delta}$ set in $\Tor$.  This is essentially contained in~\cite[Remark~2.8]{ChSh1}, for the completeness we present the complete argument to here.

For $d\geqslant 3$ and a  prime number $p$ with $\gcd(d, p-1)=1$,  the map: $x\to  x^{d}$ permutes $\F_p$. 
Hence, for any $\lambda \in \F_p\setminus\{0\}$ we have 
\begin{align*}
\sum_{n=0}^{p-1}\ep\(\sum_{j=1}^d \binom{d}{j} \lambda^j   n^j\) &
=  \sum_{n=0}^{p-1}\ep\((\lambda n+1)^{d}-1\)\\
& = \sum_{n=0}^{p-1}\ep\(n^{d}-1\)=  \sum_{n=0}^{p-1}\ep\(n\)=0.
\end{align*}
It follows that for any $\lambda \in \F_p\setminus\{0\}$ we have   
\begin{equation}
\label{eq:lambdazero}
p^{-1}\left(\binom{d}{1} \lambda^1, \ldots, \binom{d}{d} \lambda^d \right)\in \cZ_{d, W}.
\end{equation}

We call the set $\fU\subseteq \F_p^{d} $ a discrete box with 
 the side length
$
\ell(\fU) = L
$
if  
$$
\fU=\cI_1\times \ldots \times \cI_d \subseteq \F_p^{d}
$$
where the set $\cI_j = \{k_j+1, \ldots, k_j +L\}$ is a set of $L \leqslant p $ consecutive integers,
(reduced modulo $p$ if $k_j +L \ge p$) for each $j =1, \ldots, d$. 

Assuming $p > d$ we see that 
$$
 \binom{d}{j}  \not \equiv 0 \pmod p, \qquad j =1, \ldots, d.
$$
By~\cite[Lemma~2.5]{ChSh1} for any  discrete box $\fB  \subseteq \F_p^{d}$ with  side length $\ell(\fB) \geqslant C p^{1-1/2d}\log p$
for some constant $C$ there exists $\lambda\in \F_p\setminus \{0\}$ such that 
$$
\(  \binom{d}{1} \lambda , \ldots,\binom{d}{d}  \lambda^{d} \)\in \fB.
$$ 
Combining with~\eqref{eq:lambdazero} we deduce that  for any box $\cB\subseteq \Tor$ with side length larger than $5Cp^{-1/2d} \log p$  there exists a point $\vx\in \cB \cap \cZ_{d, W}$. Combining with the well known fact that there are infinitely many primes $p$  such that $(d, p-1)=1$ (for instance this follows  by applying Dirichlet's theorem on arithmetic progressions), we conclude that for any $k\in \N$ the set
$$
\bigcup_{\substack {p \geqslant k \\ p\text{ prime} \\ \gcd(d, p-1)=1} }\{\va/p:~\va \in \F_p^{d}, \,S_{d}(\va/p; p)=0\}
$$
is a dense subset of $\Tor$.  By using the continuity of the function $S_d(\vx; N)$ and applying the similar arguments as for  the sets  $ \cZ_{2, W},  \cZ_{2, G}$ and $ \cZ_{d, G}$ (with $d\ge 3$), we obtain the result. 

\subsection{Proof of Theorem~\ref{thm:H-dim}}

We first note that the lower bound for  $\dim_H \cZ_{d, W}$ with $d\ge 3$ follows from the fact that 
$$
\cZ_{2, W}\times \{0\}^{d-2}\subseteq \cZ_{d, W},
$$
and  the monotonicity of the Hausdorff dimension 
$$
\dim_H \cZ_{d, W}\ge \dim_{H} \cZ_{2, W}\times \{0\}^{d-2}.  
$$
Furthermore directly from  the definition of the Hausdorff dimension, we see that
$$
 \dim_{H} \cZ_{2, W}\times \{0\}^{d-2}=\dim_H \cZ_{2, W}.
$$

Thus in the following we only show the lower bounds on the Hausdorff dimension of the sets $\cZ_{2, W}$ and $ \cZ_{d, G}$. Our method is  a modification 
of the argument as in the proof of Theorem~\ref{thm:zero} for the sets $\cZ_{2, W}, \cZ_{2, G}$ and $\cZ_{d, G}$ (with $d\ge 3$).

In analogy of Lemma~\ref{lem:incompleteGauss} we have the following. Using the completion   techniques, see~\cite[Section~12.2]{IwKow}, similarly to  Lemma~\ref{lem:incompleteMonom}  we immediately obtain 
that or each prime $p\ge 3$ and  $a, b\in \F_p\setminus \{0\}$ with $\gcd(ab, 2p)=1$  we have 
\begin{equation}
\label{eq:P}
\max_{1\le N\le 4p}\left |\sum_{n=1}^{N}\e_{4p}(an+bn^2)\right |\ll \sqrt{p} \log p.
\end{equation}
We remark that perhaps one can also remove $\log p$ from the bound~\eqref{eq:P} and have a full analogue of  Lemma~\ref{lem:incompleteGauss}, 
but this does not affect our final result. 

Let $\cP^*$ be the collection of point $(x_1, x_2)\in\T_2$  that there are infinitely many $p$ and $a, b$ such that $\gcd(ab, 2p)=1$ and 
\begin{equation}
\label{eq:box}
0\le x_1-a/p<\frac{1}{p^{5/2}(\log p)^2}, \quad  0\le x_2-b/p<\frac{1}{p^{5/2}(\log p)^2}.
\end{equation}
For these $(x_1, x_2)$ by applying~\eqref{eq:P} and Corollary~\ref{cor:for-using} with $N=4p, K= \sqrt{p}\log p$ and $\tau = \(p^{1/2}(\log p)^2\)^{-1}$, we deduce that  
$$
 \left|\sum_{n=1}^{4p}\e(x_1n+x_2n^2) -\sum_{n=1}^{4p}\e_{4p}(an+bn^2) \right| \ll (\log p)^{-1}.
$$
By Lemma~\ref{lem:zero} (ii) we have 
$$
\left|\sum_{n=1}^{4p}\e(x_1n+x_2n^2) \right| \ll (\log p)^{-1},
$$
and hence
\begin{equation}
\label{eq:P-subset}
\cP^*\subseteq \cZ_{2, W}.
\end{equation}

For Gaussian sums let $a\in \F_p\setminus \{0\} $ and $\gcd(a, 2p)=1$.  Then by~\cite{K} we have  
\begin{equation}
\label{eq:Q}
\max_{1\le N\le 2p}\left |\sum_{n=1}^{N}\e_{2p}(an^2)\right |\ll \sqrt{p}.
\end{equation}
Let $\cQ^*$ be the collection of point $x\in [0,1)$ such that there are infinitely many $p$ and $a\in \F_p\setminus\{0\}$ with $\gcd(a, 2p)=1$ and
$$
0\le x-a/p<\frac{1}{p^{5/2}\log p}.
$$
For this $x$ and $a/p$ by~\eqref{eq:Q} and Corollary~\ref{cor:for-using} with $N=2p, K=\sqrt{p}$ and $\tau = \(p^{1/2}\log p\)^{-1}$, we obtain
$$
\left|\sum_{n=1}^{2p}\e(xn^2) -\sum_{n=1}^{2p}\e_{2p}(an^2) \right| \ll \frac{1}{\log p}.
$$
By Lemma~\ref{lem:zero} (i) we have 
$$
\left|\sum_{n=1}^{2p}\e(x_1n+x_2n^2) \right| \ll (\log p)^{-1},
$$
and hence 
\begin{equation}
\label{eq:Q-subset}
\cQ^*\subseteq \cZ_{2, G}.
\end{equation}

Now we turn to the monomial sums for $d\ge 3$. Let $\cR^*$ be the collection of point $x\in [0,1)$ such  that there are infinitely many $p$ and $a\in \F_p\setminus\{0\}$ with $\gcd(d, p-1)=1$    and 
$$
0\le x-a/p<\frac{1}{p^{d+1/2}(\log p)^2}.
$$
For this $x$ and $a/p$,  using Lemma~\ref{lem:incompleteMonom} and Corollary~\ref{cor:for-using} with $N=p$, $K=\sqrt{p}\log p$ and $\tau=(p^{1/2} (\log p)^2)^{-1}$, we derive that 
$$
\left|\sum_{n=1}^{p}\e(xn^d) -\sum_{n=1}^{p}\e_{p}(an^d) \right| \ll \frac{1}{\log p}.
$$
Combining with Lemma~\ref{lem:zero-m} we have 
$$
\left|\sum_{n=1}^{p}\e(xn^d) \right| \ll (\log p)^{-1},
$$
and hence 
\begin{equation}
\label{eq:R-subset}
\cR^*\subseteq \cZ_{d, G}.
\end{equation}

On the other hand, it  follows essentially from  the Jarn\'ik--Besicovitch theorem~\cite[Section~6]{BBDV} we obtain 
$$
\dim_H \cP^*=  6/5 \mand  \dim_H \cQ^* = 4/5, 
$$
and also for $d\ge 3$ 
$$
\dim_H \cR^*=4/(2d+1).
$$ 
Combining with~\eqref{eq:P-subset},~\eqref{eq:Q-subset} and ~\eqref{eq:R-subset} and using  the monotonicity of
the Hausdorff dimension we derive the result. We omit the details here, but refer to~\cite[Theorem~10.3]{Falconer} and~\cite{ChSh1} for the closely related arguments. 

We remark that we need the prime number theorem for the lower bounds of $ \dim_H \cP^*, \dim_H \cQ^* $, see~\cite[Theorem~10.3]{Falconer}  for details
(and also~\cite{ChSh1}).  Similarly, for the lower bounds of $\dim_H \cR^*$ we  need the prime number theorem for arithmetic progressions in a very week form that for all large enough $x$ one has 
$$
\#\{x\le p\le 2x: p \text{ is prime and } \gcd(p-1, d)=1 \}\gg \frac{x}{\log x}.
$$

\section{Further results, open problems and conjectures} 
\label{sec:further}

\subsection{An approach to improve the lower bound for $\dim_H \cZ_{d, W}$} 

We use the notation from the proof of Theorem~\ref{thm:H-dim}. For $d=2$ the  ``box condition"~\eqref{eq:box} can be extended to the below   ``rectangle  condition'' given 
by~\eqref{eq:rectangle} below.  To be precise, let $\fR$ be the collection of point $(x_1, x_2)\in\T_2$  that there are infinitely many $p$ and $a, b$ such that $\gcd(ab, 2p)=1$ and 
\begin{equation}
\label{eq:rectangle}
0\le x_1-a/p<\frac{1}{p^{3/2}(\log p)^2}, \quad  0\le x_2-b/p<\frac{1}{p^{5/2}(\log p)^2}.
\end{equation}
Clearly we have 
$$
\cP^* \subseteq \fR.
$$
Moreover for $(x_1, x_2)\in \fR$  by applying~\eqref{eq:P} and Corollary~\ref{cor:for-using} with 
$$
N=4p, \qquad K=p^{1/2}\log p,\qquad  \tau=1/p^{1/2}(\log p)^2,
$$ 
we deduce that  
$$
 \left|\sum_{n=1}^{4p}\e(x_1n+x_2n^2) -\sum_{n=1}^{4p}\e_{4p}(an+bn^2) \right| \ll (\log p)^{-1}.
$$
By Lemma~\ref{lem:zero} (ii) we have 
$$
\left|\sum_{n=1}^{4p}\e(x_1n+x_2n^2) \right| \ll (\log p)^{-1},
$$
and hence 
$$
\fR \subseteq \cZ_{2, W}.
$$
Thus we have a larger subset of $\cZ_{2, W}$, however we do not know how to deal with the Hausdorff dimension of $\fR$.  

Furthermore for $\dim_H \cZ_{d, W}$ with $d\ge 3$ we conjecture that one could  improve the Hausdorff dimension of the set  $\cZ_{d, W}$ by taking  a direct way instead of  the arguments at beginning of the proof of Theorem~\ref{thm:H-dim}. 

We remark that the results and techniques of~\cite{Brud,BD}  can shed some light 
on improving  Lemma~\ref{lem:con gen} and similar  ``continuity properties'' of Weyl sums, and thus may 
lead improvement to improvements of the bound of Theorem~\ref{thm:H-dim}  on the dimension of $\cZ_{d, W}$.

\subsection{The topology of Weyl sums}  
For $\vx\in \Tor $ we define the orbit of $\vx$ as 
$$
O_d(\vx)=\{S_d(\vx; N):~N\in \N\}.
$$
We remark that our interest to orbits is partially motivated by classical  works of Lehmer~\cite{Leh}, Loxton~\cite{Lox1, Lox2} 
and Forrest~\cite{Forrest1, Forrest2}, as well as  
by more recent results of  Cellarosi and   Marklof~\cite{CeMa},  Fayad~\cite{Fayad}, Greschonig, Nerurkar and Voln\'y~\cite{GNV}.
In the discrete settings, that is, for rational exponential sum, similar questions have been considered by   Demirci Akarsu~\cite{D_A1, D_A2}, Demirci Akarsu and   Marklof~\cite{D_AMa}, Kowalski and Sawin~\cite{KoSa1,KoSa2},  Ricotta and  Royer~\cite{RiRo},  Ricotta, Royer and the second author~\cite{RiRoSh}. 

From~\eqref{eq:zero-accumulate} and~\eqref{eq:infinity-accumulate} we obtain that there is a dense $\cG_{\delta}$ set of $\vx\in \Tor$ such that  the zero of $\CC$ is a accumulation point of $O_d(\vx)$, and the set $O_d(\vx)$ is unbounded (alternatively,  we can say that the infinity is an accumulation point of the set $O_d(\vx)$).  

Modelling the sums $S_d(\vx; N)$  by the sums of $N$ independent and uniform distributed  random 
complex vectors from the unit circle, it is natural to make the following:

\begin{conj}\label{con:dense}
For almost all $\vx\in \Tor$ in the sense of Lebesgue measure the orbit $O_d(\vx)$ is everywhere dense in $\CC$.
\end{conj}

We remark  that for some point $\vx\in \Tor$ the set $O_d(\vx)$  is not dense in $\CC$. For instance,
 Theorem~\ref{thm:HL}~(ii)  implies that  if the continued fraction of $x_2$ has bounded partial quotients then for any $x_1\in \T$ the set $O_2(\vx)$ is not a dense subset of $\CC$.  

We show two more examples  in the following that the set $O_2(\vx)$ have a biased distribution. From  Lemma~\ref{lem:zero} (ii) and~\eqref{eq:P}  for any integer $1\le a, b\le p-1$ with $\gcd(ab, 2p)=1$ and any $N\in \N$ we obtain
$$
\left|\sum_{n=1}^N\e_{4p}(an+bn^2)\right| \ll \sqrt{p}\log p.
$$
It follows that  the set $O_2(a/4p, b/4p)$ is bounded.  

Now we turn to another example. From  Lemma~\ref{lem:incompleteGauss}  for any $1\le a, b\le p-1$  and any $N>p$ we have 
$$
S_d((a/p, b/p); N)=\fl{N/p} S_d((a/p, b/p); p)+O(\sqrt{p}).
$$
Since $S_d((a/p, b/p); p)$ is a constant, we deduce that 
the infinity of $\CC$ is the only accumulation point of the set $O_2(a/p, b/p)$. Moreover the set $O_2(a/p, b/p)$ is contained in some tube of $\CC$ with width nearly $\sqrt{p}$, that is, recalling the definition of $\cA(\delta)$ at~\eqref{eq:A},
$$
O_2(a/p, b/p)\subseteq \cL(C\sqrt{p}),
$$ 
where $\cL$ is some line of $\CC$ and $C$ is some absolute positive constant.

\subsection{Restricted Weyl sums}

For an integer $d\ge 2$ and a real $\alpha\in (0,1)$ we denote 
$$
\cE_{d, \alpha}=\{\vx\in \Tor:~|S_d(\vx; N)|\ge N^{\alpha} \text{ for infinitely many } N\in \N\}.
$$
From~\eqref{eq:M-R}  we deduce that for any $\alpha\in (1/2,1)$ one has 
$\lambda_d(\cE_{d, \alpha})=0$, 
where  $\lambda_d$ is the $d$-dimensional Lebesgue measure.

Note that Theorem~\ref{thm:FK}  implies that $\lambda_2(\cE_{2, 1/2})=1$.

Motivated from the works on Diophantine approximation on manifold, see~\cite{BL, BV, HuYu} and 
references therein,  we consider the following very general  question. 
We remark that the below  set $\cX$ can be some fractal set.

\begin{question}
For  a given set $\cX\subseteq \Tor$ equipped with some measure  $\mu$ what can we say about 
$$
\dim_H (\cE_{d,\alpha}\cap \cX) \mand \mu(\cE_{d,\alpha}\cap \cX)
$$
provided $\cX$ has some natural geometric, algebraic or combinatorial structure? 
\end{question}

For the following special case when $\cX \subset \T_2$ is a parabola, we make:

\begin{conj}[Dimension]   
\label{con:dimension}
Let $\cX=\{(t, t^2):~t\in [0,1]\}$ then 
$$
\dim_H (\cE_{2,\alpha}\cap \cX)=\max\{\dim_{H} \cE_{2, \alpha}-1, 0\}.
$$ 
\end{conj}

\begin{conj}[Measure] 
\label{con:manifold}
Let $\cX=\{(t, t^2):~t\in [0,1]\}$ and let $\mu$ be the natural probability measure on $\cX$. Then one has 
  $\mu(\cE_{2, 1/2}\cap \cX)=1$. 
\end{conj}

It is also natural to ask similar  questions about the intersection $\cE_{d,\alpha}$ with  the moment curve, that is,   
$$
\cE_{d,\alpha}\cap  \{(t, t^{2}, \ldots, t^d):~t\in [0,1]\}.
$$

\subsection{Random Cantor sets} 
Now we turn to the case where $\cX$ is some random fractal set. We consider $d=2$ and a simpler model of random Cantor sets to show our ideas. We start by informal description of  the model, see~\cite[Chapter~15]{Falconer} for more related  constructions (see also~\cite{Ch} for the detailed construction and reference therein).  
 
We remark that many other random fractals are also called {\it random Cantor sets\/}.
 
We apply the following iterative procedure:

 \begin{itemize}
\item  We divide the unit square $[0,1]^2$ into four equal interior disjoint closed squares in a natural way such that each of these four squares has side length $1/2$,
We  choose uniformly at random remove one square, and let $\cE_1^{\omega}$ be the collection of the three  remaining  squares.

\item For each of the remaining  squares, we apply the same procedure and obtain a collection $\cE_2^{\omega}$   of nine squares. 
  
\item We continue inductively in the same manner by dividing each square into 
 four squares and then uniformly and independently at randomly remove  one, 
 getting a collection    $\cE_n^{\omega}$  of $3^n$ squares.
 \end{itemize}
 
 Clearly, each square in $\cE_n^{\omega}$ has the side length $2^{-n}$.

 \begin{definition}[Random Cantor set]
\label{def:Cantor}
A random Cantor set is 
$$
\cE^{\omega}=\bigcap_{n=1}^{\infty}\cE_n^{\omega}.
$$
Let $\Omega$ be our probability space which consists of  all the possible outcomes of these random limit sets. 
\end{definition}

Now for each random Cantor set $\cE^{\omega}\in \Omega$ we associate a natural  measure on $\cE^{\omega}$. The desired  measure should give each squares of $\cE_n^{\omega}$ the same mass, which is $1/3^n$. To be precise, let $\cE^{\omega}=\bigcap_{n=1}^\infty \cE_n^{\omega}$ be a realization. For each $n\in \N$ define the measure 
$$
\mu_n^{\omega}(\cA)= \int_{\T_2} \textbf{1}_{\cA \cap \cE_n^{\omega}}(\vx) (4/3)^n d\vx
$$ 
where $\textbf{1}_{\cF}$ is the indicator function of the set $\cF$. Note that for every  square $\cQ$ of $\cE_n^{\omega}$ we have 
$$
\mu_n^{\omega}(Q)=1/3^n \mand \mu_n^{\omega}(\T_2)=\mu_n^{\omega}(\cE_n)=1.
$$
Note that the sequence of the measure  $\mu_n^{\omega}, n\in \N$ weakly convergence to a measure $\mu^{\omega}$, 
see~\cite[Chapter~1]{Mattila1995}. We call this measure $\mu^{\omega}$ the natural measure on $\cE^{\omega}$. 

For our application, we  need the following Lemma~\ref{lem:measureargument}, suggested by Pablo Shmerkin (private communication). For completeness we present the complete proof here. In the following we use $\lambda_d$ to denote the $d$-dimensional Lebesgue measure. 

\begin{lemma}\label{lem:measureargument}
Let $\cF\subseteq [0,1]^d$ with $\lambda_d(\cF)=0$ then almost surely (for $\cE^{\omega}\in \Omega$)  we have $\mu^{\omega}(\cF)=0$.
\end{lemma}

\begin{proof}
We  use $\mathbb{E}$ and $\mathbb{P}$ as the notation of  expectation and probability, respectively. 

Let $\varepsilon>0$, then there is an open set $\cU\supset \cF$ with $\lambda_d(\cU)<\varepsilon$.  Applying Fubini's theorem, see~\cite[Theorem~1.14]{Mattila1995}, we obtain 
$$
\begin{aligned}
\mathbb{E} (\mu_n^{\omega}(\cU))&=\mathbb{E}  \left(\int_{\Tor} \textbf{1}_{(\cU\cap \cE_n^{\omega})}(\vx) (4/3)^nd\vx \right)\\
& =  \int_{\Tor} \mathbb{E} \left(\textbf{1}_{(\cU\cap \cE_n^{\omega})}(\vx) \right) (4/3)^nd\vx \\
&=\int_{\Tor} \mathbb{P}(\vx\in \cU\cap \cE_n^{\omega}) (4/3)^nd\vx\\
&= \lambda_d(\cU).
\end{aligned}
$$
The last identity holds by using the fact that for all $\vx\in \cU$ (with an exceptional set with zero Lebesgue measure, except some grid line) one has  
 $$
 \mathbb{P}(\vx\in \cU\cap \cE_n^{\omega})= \mathbb{P}(\vx\in \cE_n^{\omega})=(3/4)^n. 
 $$
Moreover by~\cite[Theorem~1.24]{Mattila1995} we have   
$$\mu^{\omega}(\cU)\le \liminf_{n\to  \infty}\mu_n^{\omega}(\cU).
$$
 Putting all together and applying  Fatou's  lemma~\cite[Theorem~1.17]{EG},  
 we derive 
$$
\mathbb{E}(\mu^{\omega}(F))\le\mathbb{E}(\mu^{\omega}(\cU))\le \liminf_{n\to  \infty}\mathbb{E} (\mu_n^{\omega}(\cU)) \le \lambda_d(\cU) <\varepsilon. 
$$
By the arbitrary choice of $\varepsilon>0$ we finish the proof.
\end{proof}

Combining  Lemma~\ref{lem:measureargument}  with Theorem~\ref{thm:FK} we obtain

\begin{cor}
\label{cor:Weyl-Cantor}
Almost surely for $\cE^{\omega}\in \Omega$ and  for $\mu^{\omega}$-almost all $\vx\in \cE^{\omega}$ we have 
$$
\uplim_{N\to  \infty} \frac{ |S_2(\vx; N)}{\sqrt{N} g(\ln  N)}<\infty \quad \Longleftrightarrow \quad \sum_{n=1}^{\infty} \frac{1}{ g(n)^{6}} <\infty.
$$
\end{cor}

We remark that Lemma~\ref{lem:measureargument}  indeed holds for a variety families of random fractals, 
and hence the Corollary~\ref{cor:Weyl-Cantor} also follows immediately. 


%
%
%
%
%
%

\section*{Acknowledgement}

We would like to thank Jens Marklof for informing us on some bounds of quadratic Weyl sums and also  
application of a  dynamical system approach  to the distribution of quadratic Weyl sums. 
We are grateful to Dzmitry Badziahin  and Mumtaz Hussain for helpful discussions  and some additional 
references. We also thank Pablo Shmerkin for  the idea of  Lemma~\ref{lem:measureargument}. 

This work was  supported   by ARC Grant~DP170100786.

\end{document}